\documentclass[11pt,reqno]{amsart}

\usepackage{amssymb,amsmath,epsfig}
\usepackage{amsfonts}
\usepackage{pst-plot}

\newcommand{\Z}{\mathbb{Z}}
\newcommand{\R}{\mathbb{R}}
\newcommand{\C}{\mathbb{C}}

\def\H{\mathbb{H}}
\newcommand{\leg}[2]{\genfrac{(}{)}{}{}{#1}{#2}}
\newtheorem{theorem}{Theorem}[section]

\newtheorem{lemma}[theorem]{Lemma}
\newtheorem{corollary}[theorem]{Corollary}

\newtheorem*{theorem*}{Theorem}

\numberwithin{equation}{section}

\title{Overpartitions and class numbers of binary quadratic forms}

\date{\today}
\author{Kathrin Bringmann and Jeremy Lovejoy}
\address{School of Mathematics\\University of Minnesota\\ Minneapolis, MN 55455 \\U.S.A.}
\email{bringman@math.umn.edu}
\address{CNRS, LIAFA, Universit\'e Denis Diderot,
2, Place Jussieu, Case 7014, F-75251 Paris Cedex 05, FRANCE}
\email{lovejoy@liafa.jussieu.fr}

\begin{document}
\begin{abstract}
We show that the Zagier-Eisenstein series shares its
non-holomorphic part with certain weak Maass forms whose
holomorphic parts are generating functions for overpartition rank
differences.  This has a number of consequences, including exact
formulas, asymptotics, and congruences for the rank differences as
well as $q$-series identities of the mock theta type.
\end{abstract}

\maketitle

\section{Introduction and statement of results}
In this paper we relate overpartition analogues of Ramanujan's
mock theta function $f(q)$ to the generating function for Hurwitz
class numbers $H(n)$ of binary quadratic forms of discriminant
$-n$.  The generating function for $H(n)$ is the holomorphic part
of the Zagier-Eisenstein series $\mathcal{F}(z)$ \cite{HZ,Za1},
where
\begin{equation}
\mathcal{F}(z) :=- \frac{1}{12} + \sum_{n \geq 1 \atop n \equiv
0,3 \pmod{4}} H(n)q^n +
\frac{(1+i)}{16\pi} \int_{-\overline{z}}^{i \infty}
\frac{\Theta(\tau)}{(z+\tau)^{\frac32}}\,d\tau.
\end{equation}
Here and throughout $q=e^{2 \pi iz}$ and $z=x+iy$.
The series $\mathcal{F}(z)$ transforms like a weight $\frac32$
modular form on $\Gamma_0(4)$, but it is non-holomorphic. It is
the original example of a class of functions now called \emph{weak
Maass forms} \cite{Br-Fu1} (see Section \ref{ProofSection} for the
their definition).

Building on an idea of Zwegers \cite{Zw}, the first author and Ono
\cite{Br-On1,Br-On2} have recently constructed infinite families
of weak Maass forms arising from the generating function for
Dyson's rank. Recall that Dyson \cite{Dy} defined the
\textit{rank} of a partition to be the largest part minus the
number of parts.  A special case of the results in \cite{Br-On1}
says that Ramanujan's mock theta function
\begin{equation*}
\begin{split}
f(q)=1+\sum_{n=1}^{\infty}\alpha(n)q^n:&=1+\sum_{n=1}^{\infty}\frac{q^{n^2}}{(1+q)^2(1+q^2)^2\cdots (1+q^n)^2}\\
&=1+q-2q^2+3q^3-3q^4+3q^5-5q^6+\cdots,
\end{split}
\end{equation*}
which counts the number partitions with even rank minus the number
of partitions with odd rank, is the holomorphic part of a weight
$\frac12$ weak Maass form.

Viewing this function in the framework of weak Maass forms has led
to many applications, including an exact formula for the
coefficients of $f(q)$ \cite{Br-On1}, asymptotics for the number
of partitions of $n$ with fixed rank $m$ \cite{Br}, and identities
of for rank differences \cite{BOR}.  For example, in \cite{Br-On1}
the first author and Ono proved an exact formula for $\alpha(n)$
conjectured by Andrews and Dragonette \cite{An,Dr},
\begin{equation}\label{conj}
\alpha(n)=\frac{\pi}{(24n-1)^{\frac{1}{4}}} \sum_{k=1}^{\infty} \frac{
(-1)^{\lfloor
\frac{k+1}{2}\rfloor}A_{2k}\left(n-\frac{k(1+(-1)^k)}{4}\right)}{k}
\cdot I_{\frac{1}{2}}\left(\frac{\pi \sqrt{24n-1}}{12k}\right).
\end{equation}
Here $A_{2k}(n)$ denotes a Kloosterman-type sum
and $I_{\frac12}(x)$ is the usual Bessel function of order $\frac12$.

In this paper we consider two analogues of Ramanujan's mock theta
function $f(q)$ in the setting of overpartitions.  For the
coefficients of these two series we give exact formulas which are
of a completely different nature than (\ref{conj}).  Namely, we
will exhibit formulas in terms of the Hurwitz class number $H(n)$.


Recall that an \textit{overpartition} of $n$ is a partition in
which the first occurrence of a number may be overlined.  For
example, the $14$ overpartitions of $4$ are
\begin{center}
$4$, $\overline{4}$, $3+1$, $\overline{3} + 1$, $3 +
\overline{1}$, $\overline{3} + \overline{1}$, $2+2$, $\overline{2}
+ 2$, $2+1+1$, $\overline{2} + 1 + 1$, $2+ \overline{1} + 1$,
$\overline{2} + \overline{1} + 1$, $1+1+1+1$, $\overline{1} + 1 +
1 +1$.
\end{center}
Next recall the $M2$-rank, which was introduced by the second
author \cite{Lo2} based on the work in \cite{Be-Ga1}. To define
it, we use the notation $\ell(\cdot)$ to denote the largest part
of an object, $n(\cdot)$ to denote the number of parts, and
$\lambda_o$ for the subpartition of an overpartition consisting of
the odd non-overlined parts.  Then the $M2$-rank of an
overpartition $\lambda$ is
$$
\text{$M2$-rank}(\lambda) := \bigg \lceil \frac{\ell(\lambda)}{2}
\bigg \rceil - n(\lambda) + n(\lambda_o) - \chi(\lambda),
$$
where $\chi(\lambda) = 1$ if the largest part of $\lambda$ is odd
and non-overlined and $\chi(\lambda) = 0$ otherwise.  For example,
the $M2$-rank of the overpartition
$5+\overline{4}+4+\overline{3}+1+1$ is $3-6+3-1 = -1$.

Now let $\overline{p}_e(n)$ (resp. $\overline{p}_o(n)$, $M2_e(n)$,
$M2_o(n)$) denote the number of overpartitions of $n$ with even
rank (resp. odd rank, even $M2$-rank, odd $M2$-rank). For instance
we have $\overline{p}_e(4) = 2$, $\overline{p}_o(4) = 12$,
$M2_e(4) = 6$, and $M2_o(4) = 8$. We shall be concerned with the
rank differences
\begin{eqnarray*}
\overline{\alpha}(n)&:=&\overline{p}_e(n) -
\overline{p}_o(n),\\
\overline{\alpha}_2(n)&:= & M2_e(n) - M2_o(n),
\end{eqnarray*}
whose generating functions are
\begin{eqnarray*}
\overline{f}(q) &:= & \sum_{n =0}^{\infty} \overline{\alpha}(n) \,
q^n = 1+2q-4q^2+8q^3-10q^4+8q^5-8q^6 +
\cdots,  \\
\overline{f}_2(q) &:=& \sum_{n =0}^{\infty}
\overline{\alpha}_2(n) \, q^n =
1+2q+4q^2-2q^4+8q^5+8q^6 + \cdots.
\end{eqnarray*}

In  \cite{Br-Lo1} the authors showed that $\overline{f}(q)$ is the
holomorphic part of a weak Maass form of weight $\frac32$ (note
the different weight from the partition case). Similar arguments
will be used to show that $\overline{f}_2(q)$ is also the
holomorphic part of a weight $\frac32$ weak Maass form.  Then we
shall see that the non-holomorphic parts of $\overline{f}(q)$ and
$\overline{f}_2(q)$ essentially match that of $\mathcal{F}(z)$. It
turns out that the modular forms resulting from cancelling these
non-holomorphic parts can also be written in terms of $H(n)$,
yielding exact formulas for $\overline{\alpha}(n)$ and
$\overline{\alpha}_2(n)$.

We use $\mathcal{H}(q)$ to denote the generating function for the
Hurwitz class numbers (i.e. the holomorphic part of
$\mathcal{F}(z))$:
\begin{equation*}
\mathcal{H}(q) := - \frac{1}{12} + \sum_{n \geq 1 \atop n \equiv
0,3 \pmod{4}}H(n)q^n =- \frac{1}{12} + \frac{1}{3}q^3 +
\frac{1}{2}q^4 + q^7 + q^8 + q^{11} + \cdots.
\end{equation*}
Moreover we let $\Theta(z):=\sum_{n=-\infty}^{\infty} q^{n^2}$ be the classical theta function.
\begin{theorem} \label{main} We have
\begin{itemize}
\item[$(i)$]
\begin{equation*} \label{maineq1}
\overline{f}(-q) = -16\mathcal{H}(q) - \frac13 \Theta^3(z),
\end{equation*}
\item[$(ii)$]
\begin{equation*} \label{maineq2}
\overline{f}_2(q) = - 8\mathcal{H}(q) + \frac13 \Theta^3(z).
\end{equation*}
\end{itemize}
\end{theorem}


Now to express the coefficients $\overline{\alpha}(n)$ and
$\overline{\alpha}_2(n)$ in terms of class numbers, we recall that
Gauss proved that if we define $r(n)$ by
$$
\sum_{n \geq 0}r(n)q^n := \Theta^3(z) = 1 + 6q + 12q^2 + 8q^3 +
6q^4 + \cdots,
$$
then we have
\begin{equation} \label{rofn}
r(n) =
\begin{cases}
12H(4n)& \text{if $n \equiv 1,2 \pmod{4}$}, \\
24H(n)& \text{if $n \equiv 3 \pmod{8}$}, \\
r(n/4)& \text{if $n \equiv 0 \pmod{4}$}, \\
0& \text{if $n \equiv 7 \pmod{8}$}.
\end{cases}
\end{equation}
The following formulas are then immediate:
\begin{corollary} \label{formula1} We have
\begin{itemize}
\item[($i$)]
\begin{equation} \label{formula1eq}
(-1)^n\, \overline{\alpha}(n) =
\begin{cases}
-4H(4n) & \text{if }n \equiv 1,2 \pmod{4}, \\
-24H(n) & \text{if }n \equiv 3 \pmod{8}, \\
-16H(n)&\text{if } n \equiv 7\pmod{8}, \\
-16H(n) - \frac{1}{3}r(n/4)&\text{if } 4 \mid n.
\end{cases}
\end{equation}
\item[($ii$)]
\begin{equation} \label{formula2eq}
\overline{\alpha_2}(n) =
\begin{cases}
4H(4n), &\text{if } n \equiv 1,2 \pmod{4}\\
0&\text{if } n \equiv 3 \pmod{8}\\
-8H(n), &\text{if } n \equiv 7\pmod{8} \\
-8H(n) + \frac{1}{3}r(n/4)&\text{if } 4 \mid n.
\end{cases}
\end{equation}
\end{itemize}
\end{corollary}
\noindent {\it Five  remarks.}

\noindent 1) Theorem \ref{main} further emphasizes the role that
rank differences play in linking partitions to automorphic forms
which are not classical modular forms.  In addition to the mock
theta function $f(q)$ described earlier, we recall that the
generating function for the number of partitions into distinct
parts with even rank minus the number of partitions into distinct
parts with odd rank gives rise to a Maass waveform
\cite{An-Dy-Hi1,Co1}.

\noindent 2) The method of proof yielding exact formulas for
$\overline{\alpha}(n)$ and $\overline{\alpha}_2(n)$ is completely
different from the one used in \cite{Br-On1} to obtain
(\ref{conj}). In \cite{Br-On1} the authors use Maass Poincar\'e
series whereas here we employ relations between non-holomorphic
parts of weak Maass forms.


\noindent 3)  The above equations imply that
$\overline{\alpha}(n)$ and $\overline{\alpha}_2(n)$ only grow
polynomially (like $n^k$ where $1/2-\epsilon < k < 1/2+ \epsilon$,
to be precise), whereas the coefficients $\alpha(n)$ grow
exponentially.

 \noindent 4)
One obvious application of \eqref{formula1eq} and
\eqref{formula2eq} is to use facts about class numbers to learn
about rank differences.  From a combinatorial perspective, there
are many surprising consequences. Even before appealing to the
vast knowledge about class numbers, there are immediate relations,
such as
$$\overline{p}_e(4n+1) - \overline{p}_o(4n+1) = M2_e(4n+1) -
M2_o(4n+1),
$$
which seem rather unlikely  given the disparate
definitions of Dyson's rank and the $M2$-rank.

\noindent 5) It would be interesting to see whether combinatorial
properties of overpartitions could be employed to obtain
information about class numbers.  Is there some natural involution
on overpartitions which changes the parity of the rank (or
$M2$-rank) but which is not defined on a subset whose size clearly
corresponds to class numbers?

In Corollaries \ref{cor1} and \ref{cor2} below we give a brief
indication of what can be said about overpartitions by combining
class number formulas with \eqref{formula1eq} and
\eqref{formula2eq}. Corollary \ref{cor1} contains some exact
formulas at $p^{2k}$ and $2p^{2k}$, while Corollary \ref{cor2}
gives a couple of congruences in arithmetic progressions whose
members are highly divisible by powers of $4$. We state these for
$\overline{f}(q)$, but of course similar statements hold for
$\overline{f}_2(q)$, and of course these are just a few among
endless possibilities.

\begin{corollary} \label{cor1}
If $p$ is an odd prime, then we have
\begin{equation} \label{cor1part1}
\overline{p}_e(p^{2k}) - \overline{p}_o(p^{2k})  =
\begin{cases}
2p^k, & p \equiv 1 \pmod{4}, \\
\frac{2p^k(p+1) - 4}{p-1}, & p \equiv 3 \pmod{4},
\end{cases}
\end{equation}
\begin{equation} \label{cor1part2}
\overline{p}_e(2p^{2k}) - \overline{p}_o(2p^{2k})   =
\begin{cases}
-4p^k, & p \equiv 1,3 \pmod{8}, \\
\frac{-4p^k(p+1) + 8}{p-1}, & p \equiv 5,7 \pmod{8}.
\end{cases}
\end{equation}
\end{corollary}

\begin{corollary} \label{cor2}
We have the following congruences:
\begin{itemize}
\item[$(i)$] For any $a \geq 1$, if $M \mid 2^{a + 2} - 3$, then
for all $n \geq 0$ and $t=1$ or $2$ we have
\begin{equation*}
\overline{p}_e(4^{a+1}n + t4^{a}) \equiv \overline{p}_o(4^{a+1}n + t4^{a})  \pmod{M}.
\end{equation*}
\item[$(ii)$] For any $a \geq 1$, if $M \mid 2^{a+1}-1$, then for
all $n \geq 0$ we have
\begin{equation*}
\overline{p}_e(2\cdot4^{a+1}n + 3\cdot 4^{a}) \equiv \overline{p}_o(2\cdot4^{a+1}n + 3\cdot 4^{a})
\pmod{M}.
\end{equation*}
\end{itemize}
\end{corollary}
For example, if we take $a = 4$ and $t=1$ in $(i)$ of Corollary,
then we obtain
$$
\overline{p}_e(1052n+256) \equiv \overline{p}_o(1052n+256) \pmod{61}.
$$

Another application is to relate overpartition rank differences to
other partition-theoretic functions which have connections to
class numbers.  To give one example, using work of Ono and Sze
\cite{On-Sz1} we may deduce the following:
\begin{corollary} \label{cor3}
Let $C_4(n)$ denote the number of partitions of $n$ which are
$4$-cores. If $8n+5$ is square-free, then we have
\begin{eqnarray} \label{cor3eq}
8C_4(n) &=& \overline{p}_e(8n+5) - \overline{p}_o(8n+5) \\
&=& M2_e(8n+5) - M2_o(8n+5). \nonumber
\end{eqnarray}
\end{corollary}
Thus all of the congruences and identities for $C_4(n)$ in
\cite{On-Sz1} apply to $\overline{p}_e(8n+5) -
\overline{p}_o(8n+5)$ and $M2_e(8n+5) - M2_o(8n+5)$ as well.

As a final application of Theorem \ref{main}, we use known
generating functions for $\overline{f}(q)$, $\overline{f}_2(q)$,
and class numbers to deduce $q$-series identities.  We give four
examples.
\begin{corollary} \label{cor4}
We have
\begin{itemize}
\item[($i$)]
\begin{equation} \label{cor4eq1}
\frac{4}{\Theta(z+1/2)}\sum_{n \in \mathbb{Z}}
\frac{(-1)^nq^{n^2+n}}{(1+q^{n})^2} - \frac{8}{\Theta(z)} \sum_{n
\geq 0} \frac{nq^{n^2}(1-q^{2n})}{(1+q^{2n})} = \Theta^3(z+1/2).
\end{equation}
\item[($ii$)]
\begin{equation} \label{cor4eq2}
\frac{4}{\Theta(z+1/2)}\sum_{n
\in \mathbb{Z}} \frac{(-1)^nq^{n^2+2n}}{(1+q^{2n})^2} +
4\prod_{n=1}^{\infty} \frac{(1-q^{4n-2})}{(1-q^{4n})} \sum_{n \geq 0}
\frac{q^{n^2+3n+1}}{(1-q^{2n+1})^2} = \Theta^3(z).
\end{equation}
\item[$(iii)$]
\begin{equation} \label{cor4eq3}
\frac{4}{\Theta(z)} \sum_{n \in \mathbb{Z}}
\frac{(-1)^nq^{n^2+n}}{(1+(-q)^{n})^2} - \frac{8}{\Theta(z+1/2)}
\sum_{n \in \mathbb{Z}} \frac{(-1)^nq^{n^2+2n}}{(1+q^{2n})^2} =
-\Theta^3(z)
\end{equation}
\item[$(iv)$]
\begin{equation} \label{cor4eq4}
\frac{1}{\Theta(z)} \sum_{n \in \mathbb{Z}}
\frac{(-1)^nq^{n^2+n}}{(1+(-q)^{n})^2} + \frac{1}{\Theta(z+1/2)}
\sum_{n \in \mathbb{Z}} \frac{(-1)^nq^{n^2+2n}}{(1+q^{2n})^2} =
-6\mathcal{H}(q).
\end{equation}
\end{itemize}
\end{corollary}
One may view the first three equations above as analogous to the
mock theta conjectures of Ramanujan \cite{Hi1} in the sense that a
linear combination of two non-modular functions becomes modular.


\section{weak Maass forms and the proof of Theorem \ref{main}}\label{ProofSection}
Here we prove Theorem \ref{main}.
Let us first recall the notion of a  weak Maass form.
If $k\in \frac{1}{2}\Z\setminus
\Z$,
$z=x+iy$ with $x, y\in \R$, then the weight $k$ hyperbolic
Laplacian is given by
\begin{equation}\label{laplacian}
\Delta_k := -y^2\left( \frac{\partial^2}{\partial x^2} +
\frac{\partial^2}{\partial y^2}\right) + iky\left(
\frac{\partial}{\partial x}+i \frac{\partial}{\partial y}\right).
\end{equation}
If $v$ is odd, then define $\epsilon_v$ by
\begin{equation}
\epsilon_v:=\begin{cases} 1 \ \ \ \ &{\text {\rm if}}\ v\equiv
1\pmod 4,\\
i \ \ \ \ &{\text {\rm if}}\ v\equiv 3\pmod 4. \end{cases}
\end{equation}
Moreover we let $\chi$ be a Dirichlet character.
 A {\it (harmonic) weak Maass form of weight $k$ with Nebentypus $\chi$ on a subgroup
$\Gamma \subset \Gamma_0(4)$} is any smooth function $g:\H\to \C$
satisfying the following:
\begin{enumerate}
\item For all $A= \left(\begin{smallmatrix}a&b\\c&d
\end{smallmatrix} \right)\in \Gamma$ and all $z\in \H$, we
have
\begin{displaymath}
g(Az)= \leg{c}{d}\epsilon_d^{-2k} \chi(d)\,(cz+d)^{k}\ g(z).
\end{displaymath}
\item We  have that $\Delta_k g=0$.
\item The function $g(z)$ has
at most linear exponential growth at all the cusps of $\Gamma$.
\end{enumerate}
Recall that in  \cite{Br-Lo1} the authors related $\overline{f}(q)$ to a weak Maass form.
To be more precise, define the function
\begin{eqnarray*}
  \overline{\mathcal{M}}(z):= \overline{f}(q)- \overline{\mathcal{N}}(z).
  \end{eqnarray*}
  with
  \begin{eqnarray*}
  \overline{\mathcal{N}}(z):= \frac{\sqrt{2}}{\pi i} \int_{-\bar z}^{i \infty}
   \frac{\Theta(\tau+ \frac12)}{(-i (\tau+z))^{\frac32}}
  \, d \tau.
  \end{eqnarray*}
Then the function $\overline{\mathcal{M}}(z)$ is a weak Maass form
of weight $\frac32$ on $\Gamma_0(16)$. We have a similar result
for the function $\overline{f}_2(q)$.  For this we let
$$
\overline{\mathcal{M}}_2(z):= \overline{f}_2(q) - \overline{\mathcal{N}   }_2(z),
$$
where
$$
\overline{\mathcal{N}}_2(z) :=  -\frac{i}{\sqrt{2} \pi} \int_{- \bar z}^{i\infty}
\frac{\Theta(\tau)}{\left(- i (\tau+z) \right)^{\frac32}}\, d\tau.
$$
We show
\begin{theorem} \label{MaassTheorem}
The function $\overline{\mathcal{M}}_2(z)$ is a weak Maass form of weight $\frac32$ on
$\Gamma_0(16)$.
\end{theorem}
\begin{proof}[Idea of proof of Theorem \ref{MaassTheorem}]
The proof is quite similar to the case of $\mathcal{M}(z)$ treated
in \cite{Br-Lo1}, so we just give an idea of the proof here.  We
first note that
\begin{equation} \label{M2ofq}
\overline{f}_2(q) = \frac{4\eta(2z)}{\eta^2(z)} \sum_{n \in
\mathbb{Z}} \frac{(-1)^nq^{n^2+2n}}{(1+q^{2n})^2},
\end{equation}
where
$$
\eta(z) := q^{1/24}\prod_{n=1}^{\infty} (1-q^n)
$$
is the usual $\eta$ function.   Next define the function
\begin{equation*}
\overline{M}_{2,r}(z)=
\overline{M}_{2,r}(q) := \frac{4 \eta(2z)}{\eta^2(z)}
\sum_{n \in  \Z}
\frac{(-1)^{n+1}\, q^{n^2}}{1+ e^{2 \pi i r}\, q^{2n}}
\end{equation*}
This function is related to $\overline{f}_2(q)$ via
\begin{equation*}
\overline{f}_2(q) = \frac{1}{2 \pi i } \frac{\partial}{\partial r} \big(\,  \overline{M}_{2,r}(q)\big).
\end{equation*}
One first determines a transformation law of  $\overline{M}_{2,r}(q)$ under inversion.
\begin{lemma} \label{HilfsFunktion}
We have
\begin{equation*}
\overline{M}_{2,r}\left(-\frac{1}{z} \right)
= - \sqrt{2} i (-i z)^{ \frac12}  \overline{O}_{2,-irz}(z)
- 4 \sqrt{2} (-iz)^{-\frac12} I_r(z),
\end{equation*}
where
\begin{eqnarray*}
\overline{O}_{2,r}(z)&:=& \frac{\eta(4z)}{\eta^2(8z)}
\sum_{\substack{   m \in \Z\\ m \text{ odd}  } }
\frac{q^{\frac{m^2}{2} }}{1-e^{\pi ir }\, q^{2m}},\\
I_r(z) &:=& \int_{\R} \frac{e^{- \frac{2 \pi i x^2}{z} }}{1-e^{2 \pi i r- \frac{4 \pi i x}{z}} }\, dx.
\end{eqnarray*}
\end{lemma}
Now let
\begin{equation*}
I(z):=\frac{1}{2 \pi i} \frac{\partial}{\partial r}\left( I_r(z)\right).
\end{equation*}
We relate $I(z)$ to a theta integral.
\begin{lemma} \label{ThetaLemma}
We have
\begin{equation*}
I(z) =\frac{z^2}{8 \sqrt{2}\pi} \int_{0}^{\infty}
\frac{\Theta \left( \frac{iu}{4}\right)}{   \left( -i \left(iu +z \right)\right)^{ \frac32}   }\, du.
\end{equation*}
\end{lemma}
Differentiating Lemma \ref{HilfsFunktion} and using Lemma \ref{ThetaLemma} yields
\begin{lemma}
We have
\begin{equation*}
\overline{f}_2\left(-\frac{1}{z} \right)
= -\frac{(-iz)^{\frac32 }}{\sqrt{2}} \overline{O}_2\left(\frac{z}{8} \right)
+ \frac{(-iz)^{\frac32 }}{4} \int_{0}^{\infty}
\frac{\Theta \left(\frac{iu}{4} \right)}{\left( - i (iu+z)\right)}\, du,
\end{equation*}
where
\begin{equation*}
\overline{O}_2(z) := \frac{\eta(4z)}{\eta^2(8z)}
\sum_{ \substack{m \in \Z\\ m \text{ odd} }}
\frac{q^{  \frac{m^2}{2}+2m  }}{\left( 1-q^{2m}\right)^2}
\end{equation*}
\end{lemma}
It is not hard to see that $\mathcal{N}_2(z)$ introduces the same error integral as $\overline{f}_2(q)$ under inversion.
Now one can finish the proof as in \cite{Br-On1}.
\end{proof}

We next turn to the proof of Theorem \ref{main}. For this define
the function
\begin{equation*}
\overline{g}(z): = \overline{f}(-q) + 16 \mathcal{H}(z).
\end{equation*}
We must show that $\overline{g}(z)=-\frac13 \Theta^3(q)$.

We start by showing that $\overline{g}(z)$ is a weakly holomorphic
modular form on $\Gamma_0(16)$  with trivial Nebentypus character.
For this we observe that $\overline{\mathcal{M}} \left(z+ \frac12
\right)$ is a weak Maass form of weight $\frac32$ on
$\Gamma_0(16)$ whose holomorphic part is given by
$\overline{f}(-q)$. Thus $\overline{g}(z)$ is the holomorphic part
of a weak Maass form. We now check that the associated weak Maass
form doesn't have a non-holomorphic part.
Indeed, the non-holomorphic part of
$\overline{\mathcal{M}}\left(z+\frac12 \right)$ is given by
\begin{equation} \label{thetint}
\frac{-\sqrt{2}}{\pi i} \int_{- \bar z}^{i \infty}
\frac{\Theta(\tau)}{\left( -i \left(\tau+z
\right)\right)^{\frac32}}\, d \tau.
\end{equation}
This agrees with the non-holomorphic part of $-16 \mathcal{H}(z)$.

We next claim that $\overline{g}(z)$ is actually a holomorphic
modular form.  It is known (see \cite[Prop 1.2.4]{DS}, for
example) that if the coefficients of a weakly holomorphic modular
form grow at most polynomially, then it is holomorphic.  To take
advantage of this fact in the case of $\overline{g}(z)$, we
multiply by $\Theta(z)$. Then we have
\begin{equation} \label{weight2}
\overline{g}(z)\Theta(z) = 4\sum_{n \in \mathbb{Z}}
\frac{(-1)^nq^{n^2+n}}{(1+(-q)^n)^2} + 16\mathcal{H}(z)\Theta(z),
\end{equation}
the first term on the right hand side following from the
generating function \cite{Br-Lo1}
\begin{equation} \label{fofq}
\overline{f}(q) = \frac{4}{\Theta(z+1/2)}\sum_{n \in \mathbb{Z}}
\frac{(-1)^nq^{n^2+n}}{(1+q^n)^2}.
\end{equation}
Now the coefficients in this first term of \eqref{weight2} can be
expressed in terms of sum-of-divisor functions, and hence grow at
most polynomially.  It is well-known that the Hurwitz class
numbers grow polynomially, thus so do the coefficients in the
product of $\mathcal{H}(z)$ and $\Theta(z)$.

Hence we have that $\overline{g}(z)\Theta(z)$ is a holomorphic
modular form of weight $2$ and level $\Gamma_0(16)$.  To show that
it is equal to $\frac{-1}{3}\Theta^4(z)$ we compute that the
$q$-expansions agree up to $q^4$.  This then completes the proof
of part $(i)$ of Theorem \ref{main}.

The proof of part $(ii)$ is essentially the same.  There we use
Theorem \ref{MaassTheorem} to cancel the non-holomorphic parts of
$\mathcal{\overline{M}}_2(z)$ and $\mathcal{F}(z)$, and for the
polynomial growth we appeal to the generating function for the
$M2$-rank  (\ref{M2ofq}).

\section{Proof of the corollaries}
\begin{proof}[Proof of Corollary \ref{cor1} and \ref{cor2}]
These are simple calculations using Corollary \ref{formula1}, equation \eqref{rofn},
and the following fact \cite[p.273]{Co.5}:  If $-n = Df^2$ where $D$ is a fundamental discriminant,
then
\begin{equation} \label{Cohenformula}
H(n) = \frac{h(D)}{w(D)} \sum_{d \mid f} \mu(d) \left( \frac{D}{d}
\right) \sigma_1(f/d).
\end{equation}
Here $h(D)$ is the class number of $\mathbb{Q}(\sqrt{D})$, $w(D)$
is half the number of units in the ring of integers of
$\mathbb{Q}(\sqrt{D})$, $\sigma_1(n)$ is the sum of the divisors
of $n$, and $\mu(n)$ is the M\"obius function.
\end{proof}

\begin{proof}[Proof of Corollary \ref{cor3}.]
It is shown in \cite{On-Sz1} that if $8n+5$ is square-free, then
$$C_4(n) = \frac12 h(-32n-20).$$
Invoking
\eqref{Cohenformula} shows that in this case $\overline{p}_e(8n+5)
- \overline{p}_o(8n+5)$ is also equal to $\frac12 h(-32n-20)$, thus
establishing \eqref{cor3eq}.
\end{proof}

\begin{proof}[Proof of Corollary \ref{cor4}]
Kronecker \cite{Kr1}, Mordell \cite{Mo1}, and others have given
nice generating functions involving certain class numbers $F(n)$
and $G(n)$.  The definitions of these class numbers are not
important here, only that $H(n) = G(n) - F(n)$ and
$r(n) = 24F(n) - 12G(n)$.  Thus we may recast the
main theorems in terms of these class numbers,
$$
\overline{f}(q) = -16\sum_{n \geq 0} F(n)(-q)^n  + \Theta^3(z+1/2)
$$
and
$$
\overline{f}_2(q) = -8\sum_{n \geq 0} F(n)q^n + \Theta^3(z).
$$
Then, \eqref{cor4eq1} follows from \eqref{fofq} and Mordell's
generating function for $F(n)$ \cite{Mo1}, while \eqref{cor4eq2}
follows from  \eqref{M2ofq} and Kronecker's generating function for
$F(n)$ \cite[Eq. (XI)]{Kr1}.  Equations \eqref{cor4eq3} and
\eqref{cor4eq4} follows from \eqref{fofq} and  \eqref{M2ofq}.
\end{proof}



\begin{thebibliography}{99}
\bibitem{An} G. Andrews, \emph{On the theorems of
Watson and Dragonette for Ramanujan's mock theta functions}, Amer.
J. Math.  \textbf{88} (1966), pages 454-490.
\bibitem{An-Dy-Hi1}
G.E. Andrews, F.J. Dyson, and D. Hickerson, Partitions and
indefinite quadratic forms, \emph{Invent. Math.} {\bf 91} (1988),
391-407.
\bibitem{Be-Ga1}
A. Berkovich and F.G. Garvan, Some observations on Dyson's new
symmetries of partitions, \emph{J. Comb. Theory, Ser. A} {\bf 100}
(2002), 61-93.
\bibitem{Br} K. Bringmann, \emph{Asymptotics for rank partition functions},
\emph{Trans. Amer. Math. Soc.}, accepted for publication.
\bibitem{Br-Lo1}
K. Bringmann and J. Lovejoy, Dyson's rank, overpartitions, and
weak Maass forms, \emph{Int. Math. Res. Not.} (2007), rnm063.
\bibitem{Br-On1}
K. Bringmann and K. Ono, The $f(q)$ mock theta function conjecture
and partition ranks,\emph{Invent. Math.} {\bf 165} (2006),
243-266.
\bibitem{Br-On2}
K. Bringmann and K. Ono, Dyson's rank and weak Maass forms,
\emph{Ann. Math.}, to appear.
\bibitem{BOR} K. Bringmann, K. Ono, and R. Rhoades, \emph{Eulerian series as modular form},
Journal of the American Mathematical Society, accepted for
publication.
\bibitem{Br-Fu1}
J.H. Bruinier and J. Funke, On two geometric theta lifts,
\emph{Duke Math. J.} {\bf 125} (2004), 45-90.
\bibitem{Co.5}
H. Cohen, Sums involving the values at negative integers of
$L$-functions of quadratic characters, \emph{Math. Ann.} {\bf 217}
(1975), 271-285.
\bibitem{Co1}
H. Cohen, $q$-identities for Maass waveforms, \emph{Invent. Math.}
{\bf 91} (1988), 409-422.
\bibitem{Dr} L. Dragonette,
\emph{Some asymptotic formulae for the mock theta series of
Ramanujan}, Trans. Amer. Math. Soc. \textbf{72} (1952), pages
474-500.
\bibitem{DS} F. Diamond and J. Shurman,
\emph{A first course in modular forms},
Graduate Texts in Mathematics \textbf{228}, Springer-Verlag, New York, 2005.
\bibitem{Dy} F. Dyson, \emph{Some guesses in the theory of
partitions}, Eureka (Cambridge) \textbf{8} (1944), pages 10-15.
\bibitem{Hi1}
D. Hickerson, A proof of the mock theta conjectures,
\emph{Invent. Math.} {\bf 94} No.3 (1988), 639-660.
\bibitem{HZ}
F. Hirzebruch and D. Zagier, Intersection numbers of curves on Hilbert modular surfaces
and modular forms of Nebentypus, \emph{Invent. Math.} {\bf 36} (1976), 57-113.
\bibitem{Kr1}
M. Kronecker, \"Uber die Anzahl der verschiedenen Classen
quadratischer Formen von negativer Determinante \emph{J. reine
Angew. Math.} {\bf 57} (1860), 248-255.
\bibitem{Lo2}
J. Lovejoy, Rank and conjugation for a second Frobenius
representation of an overpartition, \emph{Ann. Comb.}, to appear.
\bibitem{Mo1}
L.J. Mordell, Note on class relation formulae, \emph{Mess. Math.}
{\bf 45} (1915), 76-80.
\bibitem{On-Sz1}
K. Ono and L. Sze, 4-core partitions and class numbers, \emph{Acta
Arith.} {\bf 65} (1997), 249-272.
\bibitem{Za1}
D. Zagier, Nombres de classes et formes modulaires de poids $3/2$,
\emph{C.R. Acad Sci. Paris S\'er. A-B} {\bf 281} (1975), 883-886.
\bibitem{Zw} S. P. Zwegers, \emph{Mock theta-functions and real analytic modular forms,} $q$-series
with applications to combinatorics, number theory, and physics
(Urbana, IL, 2000), 269--277, Contemp. Math., \textbf{291}, Amer.
Math. Soc., Providence, RI, 2001.
\end{thebibliography}
\end{document}